\documentclass{article}
\usepackage{amsmath}
\usepackage{amsthm}
\usepackage{amsfonts}
\usepackage{amssymb}
\usepackage{amsrefs}
\usepackage{amsthm}
\usepackage{array}
\usepackage{graphics,graphicx}
\usepackage{xcolor}
\usepackage{float}
\newcounter{dummy}

\newtheorem{theorem}[dummy]{Theorem}
\newtheorem*{theorem*}{Theorem}

\newtheorem*{cor*}{Corollary}
\newtheorem{prop}[dummy]{Proposition}
\newtheorem*{defn}{Definition}
\theoremstyle{remark}

\newtheorem*{remark*}{Remark}
\DeclareMathOperator{\ind}{ind}
\DeclareMathOperator{\den}{density}
\DeclareMathOperator{\Area}{Area}
\numberwithin{equation}{section}

\begin{document}
\title{A Discrete Four Vertex Theorem for Hyperbolic Polygons}
\author{Kyle Grant \and Wiktor Mogilski\footnote{Department of Mathematical Sciences, Utah Valley University, Orem, UT}}

\maketitle

\begin{abstract} There are many four vertex type theorems appearing in the literature, coming in both smooth and discrete flavors. The most familiar of these is the classical theorem in differential geometry, which states that the curvature function of a simple smooth closed curve in the plane has at least four extreme values. This theorem admits a natural discretization to Euclidean polygons due to O. Musin. In this article we adapt the techniques of Musin and prove a discrete four vertex theorem for convex hyperbolic polygons.
\vspace{1 pc}
\end{abstract}

\section{Introduction}

The classical four vertex theorem in differential geometry states that the curvature function of a smooth simple closed curve in the plane has at least four extreme values (the points on the curve where these occur are called \emph{vertices}). This theorem was first proved by S. Mukhopadhayaya in 1909 \cite{Mukhopadhayaya}, albeit only proving the convex case. In 1912, A. Kneser \cite{AKneser} used a projective argument to prove the general case. Later, several independent proofs of the theorem were published. A. Kneser's son H. Kneser gave his own independent direct proof ten years later in 1922, G. Herglotz proved the convex case by contradiction in 1930, and S.B. Jackson proved the theorem by categorizing all curves that have only two vertices in 1944. A much simpler proof appeared in 1985 due to R. Osserman \cite{Osserman}. Osserman's proof uses the circumcircle of the curve and deduces the number of vertices by counting how many times the circle intersects the curve.\par

The reader might wonder about variations of the above theorem in other two-dimensional geometries. In \cite{Singer}, D. Singer provides a proof of the four vertex theorem for simple closed convex curves in the hyperbolic plane by deriving it from a theorem of Ghys. In 1945, P. Scherk \cite{Scherk} observed that stereographic projection can be used to transfer problems about vertices of plane curves to problems about
curves on the sphere, thus establishing the four vertex theorem in spherical geometry.

It is interesting to note that a discrete four vertex theorem appeared about a century earlier than any of the smooth considerations. In 1813, A. Cauchy \cite{Cauchy} proved that two convex polygons, which have corresponding sides of the same length, must either have equal corresponding angles or the difference between the corresponding angles must change sign at least four times. Unfortunately, E. Steinitz found a mistake in Cauchy's proof about a hundred years later and published a correct version in 1934 (it is now known as the Cauchy-Steinitz lemma). Around this time, several other discrete four vertex theorems were published. R.C. Bose published a version in 1932 which followed from equations involving the number of empty circles and number of full circles of a polygon. A.D. Aleksandrov proved a version similar to the Cauchy-Steinitz lemma in 1950 instead considering sign changes in edges, and then S. Bilinski proved a version in 1963 which instead focused on sign changes of differences angles in a single convex polygon.

While smooth four vertex type theorems are interesting in their own right, their discrete counterparts have several advantages. For example, they are stronger than their smooth versions, implying them by passage of the limit. For another example, discrete theorems are also usually simpler to state and can usually be proved by a much easier combinatorial argument (e.g. induction).

Of particular interest to us is a notion of discrete curvature and a four vertex theorem introduced by O. Musin in \cite{Musin2}. Given a convex polygon in the plane (with some additional mild conditions that we define later) with vertices $V_{1},V_{2},...,V_{n}$, let $R_i$ denote the radius of the circle passing through the consecutive vertices $V_{i-1}V_iV_{i+1}$. We say that a vertex $V_i$ is extremal if either $R_{i-1}>R_i<R_{i+1}$ or $R_{i-1}<R_i>R_{i+1}$. Musin proves (by contradiction) that a convex polygon with at least four vertices has at least four extremal vertices. Note that this notion of discrete curvature is compatible with the standard notion of curvature in the plane. This can be seen by considering the osculating circle at any point of the smooth curve, the radius of which is the reciprocal of the curvature. Hence Musin's theorem is a discretization of the classical smooth four vertex theorem.

In \cite{Musin1}, Musin extends this notion of discrete curvature to behave well with non-convex polygons. He then introduces a new polygonal curve associated to a polygon called the \emph{discrete evolute}. This is the discretization of the evolute of a smooth curve, which is a new curve obtained from the centers of osculating circles of the original curve. It turns out that the discrete evolute can be used to detect extremal vertices. Musin exploits this fact to obtain an equality relating the discrete turning number of both the polygon and the evolute to the number of extremal vertices. A four vertex theorem is then a consequence of this equality.

In this paper we cast Musin's notion of vertex extremality into the setting of the hyperbolic plane. We then adapt his program and prove an equality relating what we call the \emph{density} of discrete hyperbolic evolute to the number of extremal vertices (Theorem 5). From this we derive a four vertex theorem for convex hyperbolic polygons (Theorem 6). This is a discrete version of a smooth four vertex theorem for hyperbolic curves with curvature everywhere greater than $1$.

\section{Discrete Curvature and Vertex Extremality}

In this paper we will restrict ourselves to the hyperbolic plane $\mathbb{H}^2$. We will begin with a few definitions and notation. By $P$ we will denote a polygonal curve in $\mathbb{H}^2$, which is simply a curve with non-ideal vertices $V_{1},V_{2},...,V_{n}$, where each consecutive pair of vertices is joined by a geodesic segment, and successive hyperbolic segments meet only at the points $V_i$. The polygonal curve $P$ is \emph{closed} if $V_1=V_n$ and we say that $P$ is \emph{simple} if it has no self intersections. For brevity, we will refer to a simple closed polygonal curve as a polygon.

In hyperbolic geometry it is possible to have three points that do not lie on the same geodesic and are not circumscribed by a hyperbolic circle. For example, the points could instead lie on a different type of curve with constant curvature such as a horocycle or a hypercycle. To ensure that the notion of a hyperbolic evolute (defined in Section 3) is a polygonal curve in the above sense, we will impose the following condition on our polygonal curves $P$. The analogous condition for smooth curves is that the curvature is everywhere greater than $1$, which is not an unusual assumption in hyperbolic geometry.

\begin{defn}
We say that a polygonal curve $P$ in $\mathbb{H}^2$ is generic if no vertices are ideal and
\begin{description}
  \item[(a)] The maximal number of vertices of $P$ that lie on a hyperbolic circle is three.
  \item[(b)] No three vertices lie on the same geodesic.
\end{description}
\end{defn}

For a generic polygonal curve $P$, let $C_{i}=C(V_{i-1}V_{i}V_{i+1})$ denote the hyperbolic circumcircle formed by the corresponding vertices of $P$, $O_{i}=O(V_{i-1}V_{i}V_{i+1})$ the center of $C_{i}$, and $R_i$ the radius of $C_i$. Note that all indices are taken modulo the number of vertices of the polygonal curve $P$.

\begin{defn}
A polygonal curve $P$ is coherent if for any three consecutive vertices $V_{i-1},V_{i},$ and $V_{i+1}$, the center of the circle $C_{i}$ lies in the infinite cone formed by the vertices $V_{i-1},V_{i},$ and $V_{i+1}$.
\end{defn}

The figure below illustrates the situation where $P$ can fail to be coherent: the mediatrices of the edges $\overline{V_{i-1}V_i}$ and $\overline{V_{i}V_{i+1}}$ intersect outside of the infinite cone formed by the three consecutive vertices.
        \begin{figure}[H]
        \centering
        \includegraphics[scale=.25]{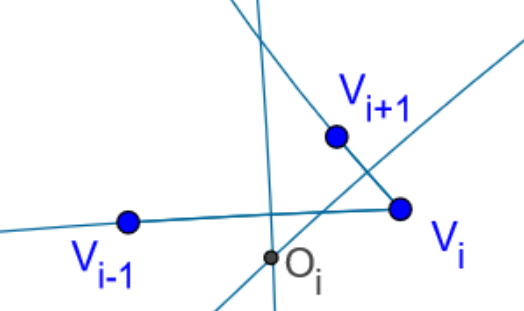}
        \caption{$P$ is not coherent}
                \label{fig:sines}
    \end{figure}

From here on, \emph{all polygonal curves will be assumed to be generic and coherent}. We will also impose the convention that we are traversing such polygons $P$ counterclockwise and we will denote the left angle at a vertex with respect to this orientation by $\angle V_i=\angle V_{i-1}V_{i}V_{i+1}$.

\begin{defn} A vertex $V_{i}$ is said to be positive if $\angle V_i$ is at most $\pi$. Otherwise, it is said to be negative. If all vertices of $P$ are positive, then $P$ is
convex.
\end{defn}

Following Musin \cite{Musin1}, we now define a notion of discrete curvature on a generic and coherent polygonal curve. Assume that a vertex $V_i$ is positive. We say that the curvature of the vertex $V_{i}$ is greater than the curvature at $V_{i+1}$ ($V_{i}\succ V_{i+1})$ if the vertex $V_{i+1}$ is positive and $V_{i+2}$ lies outside the circle $C_{i}$ or if the vertex $V_{i+1}$ is negative and $V_{i+2}$ lies inside the circle $C_{i}$.

By switching the word ``inside" with the word ``outside" in the above definition (and vice-versa), we obtain that $V_{i}\prec V_{i+1}$, or that the curvature at $V_{i}$ is less than the curvature at $V_{i+1}$.

In the case that the vertex $V_{i}$ is negative, simply switch the word ``greater" with the word ``less", and the word ``outside" by the word ``inside".

The following proposition justifies that this is in fact a reasonable discrete notion of curvature.

\begin{prop}
Let $P$ be a convex polygon.
\begin{enumerate}
  \item $V_{i-1}\prec V_{i}$ if and only if $R_{i-1}>R_i$.
  \item $V_{i-1}\succ V_{i}$ if and only if $R_{i-1}<R_i$.
\end{enumerate}
\end{prop}
\begin{proof}
As $P$ is assumed to be convex, $V_{i-1}$ and $V_{i}$ are both positive. We will only prove the first item. Assume that $V_{i-1}\prec V_{i}$ so that $V_{i+1}$ lies inside the circle $C_{i-1}$. Note that, since $P$ is coherent, $O_i$ lies on the same side of $\overleftrightarrow{V_{i-1}V_i}$ as $O_{i-1}$.

 \par
    Now, $O_i$ lies on the mediatrices of $\overline{V_{i-1}V_i}$ and $\overline{V_{i}V_{i+1}}$. Since $V_{i+1}$ is inside $C_{i-1}$, it follows that $O_i$ is inside the triangle $\Delta V_{i-1}O_{i-1}V_{i}$. Note that this is an isosceles triangle with side lengths $R_{i-1}$, $R_{i-1}$ and $V_{i-1}V_i$. We can additionally define $\Delta V_{i-1}O_iV_i$, the isosceles triangle with side lengths $R_i$, $R_i$ and $V_{i-1}V_i$. These two triangles have the same base (see the figure below).

        \begin{figure}[H]
        \centering
        \includegraphics[scale=.4]{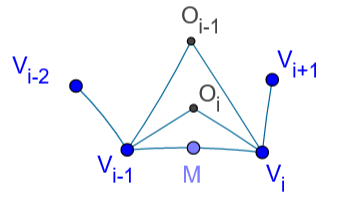}
        \caption{$V_{i-1}\prec V_{i}$}
                \label{fig:sines}
    \end{figure}

    Since $O_i$ is inside $\Delta V_{i-1}O_{i-1}V_i$, we have that $V_{i-1}\prec V_{i}$ if and only if
    \begin{equation*}
        \angle V_{i-1}V_iO_i < \angle V_{i-1}V_iO_{i-1},
    \end{equation*}
    and this happens if and only if $\angle O_i > \angle O_{i-1}$. Now consider the right triangles produced by placing a point at the midpoint $M$ of $\overline{V_{i-1}V_{i}}$. Then, by the law of sines in hyperbolic geometry,
    \begin{align*}
        \frac{\sin{\left(\frac{\angle O_{i-1}}{2}\right)}}{\sinh{(MV_i)}} & = \frac{\sin{\left(\frac{\pi}{2}\right)}}{\sinh{(R_{i-1})}} = \frac{1}{\sinh{(R_{i-1})}} \\
        \text{and }\hspace{2.5 cm}& \\
        \frac{\sin{\left(\frac{\angle O_{i}}{2}\right)}}{\sinh{(MV_i)}} & = \frac{\sin{\left(\frac{\pi}{2}\right)}}{\sinh{(R_{i})}} = \frac{1}{\sinh{(R_{i})}}. \\
    \end{align*}
    Note that since $\frac{\angle O_{i-1}}{2}$ and $\frac{\angle O_i}{2}$ are angle measures of angles in a right triangle,
    \begin{equation*}
        \frac{\angle O_{i-1}}{2} < \frac{\angle O_i}{2} \leq \frac{\pi}{2}.
    \end{equation*}
    Therefore, we have that
    \begin{align*}
        \sinh{(R_{i-1})}\sin{\left(\frac{\angle O_{i-1}}{2}\right)} & = \sinh{(R_i)}\sin{\left(\frac{\angle O_i}{2}\right)} \\
        &> \sinh{(R_i)}\sin{\left(\frac{\angle O_{i-1}}{2}\right)}.
           \end{align*}
    The above computation is equivalent to $\sinh{(R_{i-1})} > \sinh{(R_i)}$, and hence $R_{i-1} > R_i$. Thus we have shown that $V_{i-1}\prec V_{i}$ if and only if $R_{i-1} > R_i$.

\end{proof}

\begin{defn}A vertex $V_{i}$ of a polygonal line $P$ is locally maximal if

\centerline{$V_{i-1}\prec V_{i} \succ V_{i+1}$ and is locally minimal if $V_{i-1}\succ V_{i} \prec V_{i+1}$.}

\end{defn}

Note that the above notion of discrete curvature turns $P$ into a directed graph, i.e., a graph with an arrow indicating a direction on every edge.

\begin{defn}
We say that a vertex $v$ on a directed graph $G$ is locally extremal if either all edges that meet at $v$ have an arrow pointing away from $v$ (i.e. all edges exit $v$) or all edges that meet at $v$ have an arrow that points towards $v$ (i.e. all edges enter $v$).
\end{defn}

Let $l_{+}(G)$ denote the number of locally minimal vertices of $G$ (i.e. ones where all edges enter $v$) and $l_{-}(G)$ denote the number of locally maximal vertices of $G$ (i.e. vertices where all edges exit $v$).

Suppose that $G$ is planar. For a vertex $v\in G$, define the \emph{index} of $v$ to be

$$\ind(v)=1-G^{-}(v),$$

Where $G^{-}(v)$ is the number of edges exiting $v$. The following proposition implies (under the ordering defined above) that $l_+(P)=l_-(P)$.

\begin{prop}
\label{prop:maxmin}
Let $G$ be a cycle graph. Then $$l_{+}(G)=l_{-}(G).$$
\end{prop}
\begin{proof}
First, note that if $v$ is locally minimal then $\ind(v)=1$ and if $v$ is locally maximal then $\ind(v)=-1$. Otherwise, $\ind(v)=0$.  Hence, $$\sum_{v\in G}\ind(v) = l_{+}(G) - l_{-}(G).$$
We are done if we establish the fact that $\sum_{v\in G}\ind(v) = 0$, but this is a special case of the discrete Poincar\'{e}--Hopf formula \cite{Knill}[Theorem 1], which states that for a directed graph $G$ one has $\sum_{v\in G} \ind(v)=\chi(G)$ (here $\chi(G)$ denotes the Euler characteristic of $G$).
\end{proof}

\section{The Evolute of a Polygon}

In this section we will define a new (possibly non-simple) polygonal curve associated to a given hyperbolic polygon $P$ called the evolute. A salient feature of the evolute is that it can be used to detect the presence of extremal vertices.

\begin{defn}
The closed polygonal curve formed by the centers $O_{1},O_{2},...,O_{n}$ is called the evolute $E(P)$ of $P$.
\end{defn}

$E(P)$ is traversed by following the consecutive order of the vertices. We will denote the left angle at a vertex $O_i$ by $\angle O_i=\angle O_{i-1}O_{i}O_{i+1}$ (see Figure 3 and Figure 4 below).

\begin{defn}
A vertex $O_i$ of the evolute is said to be a cusp if
$$|\angle O_i - \angle V_i| > \pi  .$$
\end{defn}

\begin{theorem}
A vertex $V_i$ of $P$ is extremal if and only if $O_i$ is a cusp.
\end{theorem}

\begin{proof}
We will only consider the case where $P$ is convex (so all vertices are positive). The remaining several non-convex cases will follow a similar routine check.

Suppose that $V_i$ is maximal. Then we have the following configuration about the vertex $V_i$ (see the figure below).
    \begin{figure}[H]
        \centering
        \includegraphics[scale=.4]{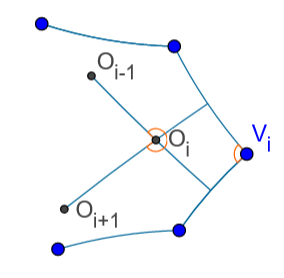}
        \caption{$V_i$ locally maximal}
        \label{fig:MaxEvolute}
    \end{figure}
    Observe that we have a quadrilateral with angle measures $\frac{\pi}{2}$, $\angle V_{i}$, $\frac{\pi}{2}$, and $2\pi - \angle O_{i}$. Since the angle sum of this quadrilateral is less $2\pi$ we have that $\angle V_{i} - \angle O_{i} < -\pi$, and hence $O_{i}$ is a cusp of $P$.

If $V_i$ is minimal, then we have the configuration illustrated in the figure below.
\begin{figure}[H]
    \centering
    \includegraphics[scale=.3]{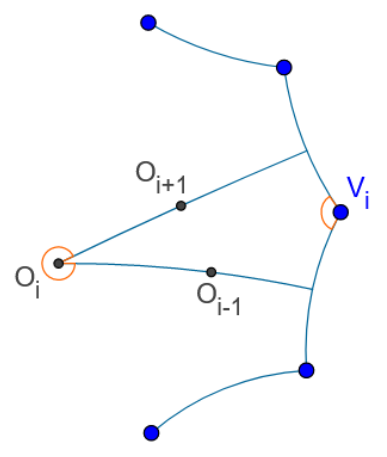}
    \caption{$V_{i}$ locally minimal}
    \label{fig:MinEvolute}
\end{figure}
We have that $\angle V_{i}+ 2\pi - \angle O_{i}<\pi$. This again implies that $\angle V_{i} - \angle O_{i} < -\pi$, and hence $O_{i}$ is a cusp of $P$.
\end{proof}

Note that $V_i$ is not extremal if and only if $|\angle O_{i} - \angle V_{i}| < \pi  $. Furthermore, if we assume that $P$ is convex, then we can omit the modulus in computations (see the proof above and the figure below). In the figure below, one has that $\angle V_{i}+ \pi+\pi - \angle O_{i}<2\pi$, and hence $0<\angle O_i - \angle V_i<\pi$.
\begin{figure}[H]
    \centering
    \includegraphics[scale=.4]{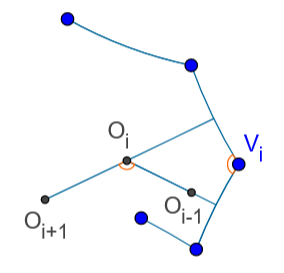}
    \caption{$V_i$ is not extremal}
    \label{fig:NonextEvolute}
\end{figure}

\section{A Discrete Hyperbolic Four Vertex Theorem}

\begin{defn}
Let $P$ be a polygonal curve with vertices $V_1,V_2,...,V_n.$ Then the polygon density of $P$, denote by $\den(P)$, is defined to be:
$$\den(P)=\frac{1}{2\pi}\sum_{i=1}^n(\pi-\angle V_{i}).$$
\end{defn}

In the Euclidean plane, the polygon density is simply equal to the winding number (also called the rotation index). In fact, for polygons in the plane, a simple triangulation argument shows that the polygon density is always equal to $1$. In contrast, the Gauss--Bonnet theorem implies that the formula for a hyperbolic polygon is $\den(P)=1 + \frac{1}{2\pi}\Area{(P)}$.

It is a well known fact that the winding number of a curve on a surface can be computed by that of its preimage with respect to the coordinate patch of the surface which contains the curve. Hence the winding number of an evolute is a negative integer when $P$ is a polygon (e.g. see the discussion in \cite{Musin1} after Theorem 3.2). We obtain the following inequality from the generalized Gauss--Bonnet theorem \cite{Cufi}[Theorem 6.1].

\begin{theorem}
Let $P$ be a polygon. Then $$\den(E(P))\leq -1.$$
\end{theorem}

We now will prove an equality relating the density of a polygon and its evolute to the number of extremal vertices of $P$. We will use it to derive a discrete four vertex theorem for convex hyperbolic polygons.

\begin{theorem}
Let $P$ be a convex polygon and let $N$ denote the number of locally extremal vertices of $P$. Then

$$2\den(P)-2\den(E(P))=N+\frac{1}{\pi}\sum_{i=1}^n\delta_i,$$
where $\delta_i$ is the defect of the quadrilateral formed by $V_i$, $O_i$ and the midpoints of $\overline{V_{i-1}V_i}$ and $\overline{V_{i}V_{i+1}}$.
\end{theorem}
\begin{proof}
Unwinding definitions, we have that
\begin{align*}
    2\den(P)-2\den(E(P)) & = \frac{1}{\pi}\sum_{i=1}^n(\pi - \angle V_i) - \frac{1}{\pi}\sum_{i=1}^n(\pi - \angle O_i) \\
    & = \frac{1}{\pi}\sum_{i=1}^n(\angle O_i - \angle V_i).
\end{align*}
By Theorem 3, $V_i$ is locally extremal if and only if $\angle O_i - \angle V_i > \pi$. In particular, $\angle O_i - \angle V_i = \pi + \delta_i$ where $\delta_i = \pi - \angle V_i - \alpha_i$ and $\alpha_i$ is the angle opposite to $\angle V_i$ in the quadrilateral in the statement of the theorem. Furthermore, $V_i$ is not extremal if and only if $0 < \angle O_i - \angle V_i < \pi$. In particular, $\angle O_i - \angle V_i = \delta_i$. Thus, by the equation above,
\begin{equation*}
    2\den{(P)} - 2\den{(E(P))} = N + \frac{1}{\pi}\sum_{i=1}^n\delta_i.
\end{equation*}
\end{proof}

We are now ready to derive our hyperbolic four vertex theorem.

\begin{theorem}
Every convex hyperbolic polygon with at least four vertices has at least four locally extremal vertices.
\end{theorem}
\begin{proof}
Let $P$ be a convex polygon with $n\geq 4$ vertices and $N$ locally extremal vertices. By Theorem 5,
\begin{equation*}
    2\den(P)-2\den(E(P))=N+\frac{1}{\pi}\sum_{i=1}^n\delta_i,
\end{equation*}
where $\delta_i = \pi - \angle V_i - \alpha_i$ (see the proof of Theorem 5).

By the definition of $\den{(P)}$, the above equation rewrites as,

   $$ \frac{1}{\pi}\sum_{i=1}^n\alpha_i -2\den{(E(P))} = N.$$

By Theorem 4, $-2\den{(E(P))} \geq 2$ and furthermore $\frac{1}{\pi}\sum_{i=1}^n\alpha_i > 0$, hence $N>2$. By Proposition 2, the number of maximal extremal vertices must be equal to the number of minimal extremal vertices. That is, $N$ must be even. Therefore $N\geq 4$.
\end{proof}

\end{document}